\documentclass[a4paper,10pt]{article}
\usepackage{amsmath}
\usepackage{amssymb}
\usepackage{amscd}
\usepackage{amsthm}
\newtheorem{theorem}{Theorem}
\newtheorem{lemma}{Lemma}
\newtheorem{corollary}{Corollary}
\newtheorem{definition}{Definition}
\newtheorem{remark}{Remark}
\newcommand{\newweight}{2p+(p-1)/2}
\newcommand{\F}{\mathbb{F}}
\title{On the code generated by the incidence matrix of points and hyperplanes in $PG(n,q)$ and its dual}
\author{M. Lavrauw\and L. Storme \and G. Van de Voorde \thanks{This author's research was supported by the                 Institute for the Promotion of Innovation through Science
and Technology in Flanders (IWT-Vlaanderen).} }
\date{\today}
\begin{document}

\maketitle
\begin{abstract} In this paper, we study the $p$-ary linear code $C(PG(n,q))$, $q=p^h$, $p$ prime, $h\geq 1$, generated by the incidence matrix of points and hyperplanes of a Desarguesian projective space $PG(n,q)$, and its dual code. We link the codewords of small weight of this code  to blocking sets with respect to lines in $PG(n,q)$ and we exclude all possible codewords arising from small linear
blocking sets. 

We also look at the dual code of $C(PG(n,q))$ and we prove that finding the minimum weight of the
dual code can be reduced to finding the minimum
weight of the dual code of points and lines in $PG(2,q)$. We present an improved upper
bound on this minimum weight and we show that we can drop the divisibility condition on the weight of the codewords in Sachar's lower bound \cite{sachar}.
\end{abstract}
\section{Introduction}
In this paper, we denote the $n$-dimensional projective space over the finite field of order $q$, where $q=p^h$, $p$ prime, $h\geq 1$, by $PG(n,q)$. Let $\theta_n$ denote the number of points in $PG(n,q)$, i.e., $\theta_n=(q^{n+1}-1)/(q-1)$, and let $V(n+1,q)$ denote the underlying vector space.

This research is a natural extension of the results on the $p$-ary linear code generated by points and lines of a projective plane $PG(2,q)$, with $q=p^h$, $p$ prime, $h\geq 1$. The minimum weight and the nature of the minimum weight codewords of the $p$-ary linear codes generated by the incidence matrix of points and lines of projective planes, have been established in the 1960s, after Prange \cite{prange} and Rudolph \cite{rudolph}  recognized that projective planes could be used to produce error-correcting codes.  The codewords of minimal weight are the scalar multiples of the incidence vectors of the lines of $PG(2,q)$ \cite[Theorem 6.3.1]{AK}.
In \cite{chouinard},  Chouinard investigates the codewords of small weight in this code. In particular, when $q$ is prime, the following result is proven.
\begin{theorem}\label{ch}\cite{chouinard} (1) In the $p$-ary linear code arising from $PG(2,p)$, $p$ prime, there are no codewords with weight in $[p+2,2p-1]$.

(2) The codewords of weight $2p$ in the $p$-ary linear code arising from $PG(2,p)$, $p$ prime, are the scalar multiples of the differences of the incidence vectors of two lines of $PG(2,p)$.
\end{theorem}

In \cite{fack}, this result was extended to codewords of larger weight in the following theorem.
\begin{theorem} \cite{fack}
The only codewords $c$, with $0< wt(c) \leq \newweight$, in the $p$-ary linear code $C$ arising from $PG(2,p)$, $p$ prime, $p\geq 11$, are:
\begin{itemize}
\item codewords with weight $p+1$: the scalar multiples of the incidence vectors of the  lines of $PG(2,p)$, 
\item codewords with weight $2p$: $\alpha(c_1-c_2)$, $c_1$ and $c_2$ the incidence vectors of two distinct lines of $PG(2,p)$,
\item codewords with weight $2p+1$: $\alpha c_1 + \beta c_2,~\beta \neq - \alpha$, with $c_1$ and $c_2$ the incidence vectors of two distinct lines of $PG(2,p)$.
\end{itemize}
\end{theorem}

Moreover, in \cite{fack}, the first part of Theorem \ref{ch} was extended to $\F_{p^3}$.
\begin{theorem}\cite{fack}
In the $p$-ary linear code of $PG(2,p^3)$, $p$ prime, $p\geq 7$, there are no codewords with weight in the interval $[p^3+2,2p^3-1]$.
\end{theorem}
\begin{remark}\label{remkw}
The same result holds for $\F_{p^2}$, $p$ prime, which can be deduced easily in the same way as the authors do in \cite{fack}. 

Namely, it is known that a codeword of weight in $]p^2+1,2p^2[$ in the $p$-ary linear code of $PG(2,p^2)$, $p$ prime, is a scalar multiple of the incidence vector of a non-trivial minimal blocking set  in $PG(2,p^2)$, $p$ prime, intersecting every line in $1\pmod{p}$ points \cite{chouinard}, \cite[Lemma 1]{fack}. The only such non-trivial minimal blocking sets are the Baer subplanes \cite{SZ}, but these are not codewords in the $p$-ary linear code of $PG(2,p^2)$ \cite[Proposition 6.6.3]{AK}.

Hence, we obtain the following result. 
\end{remark}

\begin{theorem}\label{th:1}
The $p$-ary linear code of $PG(2,p^2)$, $p$ prime, does not have codewords with weight in $[p^2+2,2p^2-1]$.
\end{theorem}

The goal of the first section of this paper is to prove similar results for general dimension $n$ and field order $q$. 

We know that the codewords of minimum weight in the $p$-ary linear code defined by the incidence matrix of points and hyperplanes of $PG(n,q)$, $q=p^h$, $p$ prime, $h\geq 1$,  are the scalar multiples of the incidence vectors of the hyperplanes of $PG(n,q)$ \cite[Proposition 5.7.3]{AK}. We will study codewords of weight in $]\theta_{n-1},2q^{n-1}[$, and  show that there is a gap in the weight enumerator of this code by excluding as many weights as possible in this interval.  More precisely, we will show that there are no codewords with weight between the weight of a hyperplane and the symmetric difference of two hyperplanes for $q=p$ and $q=p^2$, $p>11$, $p$ prime.
Corollary \ref{gevolg} proves the analogous statement of Theorem \ref{ch} (1) for general dimension.  Extending the theorem for codes $C(PG(n,q))$, over an arbitrary finite field $\mathbb{F}_q$, is harder. Here we show that a codeword of weight in $]\theta_{n-1},2q^{n-1}[$ corresponds to a minimal blocking set in $PG(n,q)$, and we exclude all linear blocking sets with weight in $]\theta_{n-1},2q^{n-1}[$ as codewords. We also prove that the weights of the codewords of weight in $]\theta_{n-1},2q^{n-1}[$ can only lie in a number of small intervals, and that there are no codewords with weight in $[3q^{n-1}/2,2q^{n-1}[$. In this way, half of the interval is eliminated. If $q$ is the square of a prime, this proves the statement of Remark \ref{remkw} and Theorem \ref{th:1} in general dimension.\\

The situation regarding the dual of the code generated by the incidence matrix of points and lines in $PG(2,q)$ is different. In this case, the minimum weight of the dual code is not known in general, although some bounds are given (see Assmus and Key \cite{AK} and Sachar \cite{sachar}).
We extend these results to general dimension by proving that the minimum weight of the dual code generated by the incidence matrix of points and hyperplanes in $PG(n,q)$ is equal to the minimum weight of the dual code generated by the incidence matrix of points and lines in $PG(2,q)$.
Moreover, we present an improved upper bound on this minimum weight and we show that we can drop the divisibility condition on the weight of the codewords in Sachar's lower bound.

\section{Small weight codewords in the code generated by the incidence matrix of points and hyperplanes in $PG(n,q)$}
In this section, we investigate the codewords of small weight in the linear code generated by the incidence matrix of points and hyperplanes in $PG(n,q)$. We define the incidence matrix~$A = (a_{ij})$ 
of the projective space $PG(n,q)$, $q=p^h$,  $p$~prime, $h\geq 1$, 
as the matrix whose rows are 
indexed by hyperplanes of the space and whose columns are indexed
by points of the space, and 
with entry
$$ 
a_{ij} = \left\{
\begin{array}{ll}
1 & \textrm{if point $j$ belongs to hyperplane $i$,}\\
0 & \textrm{otherwise.}
\end{array} 
\right.
$$

The $p$-ary linear code $C$ of the projective space $PG(n,q)$, $q=p^h$, $p$ prime, $h\geq 1$,  is 
the $\mathbb{F}_p$-span of the rows of the  incidence matrix $A$. The support of a codeword $c$, denoted by $supp(c)$, is the set of all non-zero positions of $c$. We identify this set of positions with the set of corresponding points of $PG(n,q)$. Let $c_P$ denote the symbol of the codeword $c$ in the coordinate position corresponding to the point $P$.
We denote the scalar product of two vectors $v_1,v_2$, calculated over $\F_p$, by $(v_1,v_2)$.

The dual code $C^\bot$ is the set of all vectors orthogonal to all codewords of $C$, hence
$$C^\bot=\{ v\in V(\theta_n,p) \vert \vert(v,c)= 0,\ \forall c\in C\}.$$

From now on, we denote the $p$-ary linear code of points and hyperplanes of $PG(n,q)$, $q=p^h,p$ prime, $h\geq1$, by $C$ and its dual code by $C^\bot$. If we want to point out the dimension and field of the considered space, we write $C(PG(n,q))$ and $C(PG(n,q))^\bot$, respectively.  For convenience of notation, we identify a space with its incidence vector, hence the symbol $l$ stands for the line $l$ or the incidence vector of $l$, depending on the context.

\begin{lemma}  \label{1}If $U_1$ and $U_2$ are subspaces of dimension at least $1$ in $PG(n,q)$, then $U_1-U_2\in C^\bot$.
\end{lemma}
\begin{proof} For every subspace $U_i$ of dimension at least $1$ and every hyperplane $H$, $(H,U_i)=1$, hence $(H,U_1-U_2)=0$, so $U_1-U_2\in C^\bot$.
\end{proof}
Note that in Lemma \ref{1}, $\dim U_1 \neq \dim U_2$ is allowed.
\begin{lemma} \label{lem4}The scalar product $(c,U)$, with $c\in C$ and $U$ an arbitrary subspace of dimension at least $1$, is a constant.
\end{lemma}
\begin{proof} Lemma \ref{1} yields that $U_1-U_2\in C^\bot$, for all subspaces $U_1,U_2$ with $\dim(U_i)\geq 1$, hence $(c,U_1-U_2)=0$, so $(c,U_1)=(c,U_2)$.
\end{proof}
\begin{lemma} \label{lem5} A codeword $c$ is in $C\cap C^\bot$ if and only if $(c,U)=0$ for all subspaces $U$ with $\dim(U)\geq 1$.
\end{lemma}
\begin{proof} Let $c$ be a codeword of $C\cap C^\bot$. Since $c\in C^\bot$, $(c,H)=0$ for all hyperplanes $H$, Lemma \ref{lem4} yields that $(c,U)=0$ for all subspaces $U$ with dimension at least $1$.

Now suppose that $c\in C$ and $(c,U)=0$ for all subspaces $U$ with dimension at least $1$. Applying this to a hyperplane yields that $c\in C\cap C^\bot$.
\end{proof}

\begin{theorem} \label{th11}
The minimum weight of $C\cap C^\bot$ is equal to $2q^{n-1}$.
\end{theorem}
\begin{proof}
It follows from Lemma \ref{lem5} that the support of a codeword $c$ in $C\cap C^\bot$ corresponds to a set of points such that every line contains zero or at least two of them. If $wt(c)<2q^{n-1}$, then there is a line $L$ containing exactly two points of $supp(c)$. Suppose not, then all lines through a point $P\in supp(c)$ would have two extra intersection points with $supp(c)$, which would imply that $wt(c)\geq1+2\theta_{n-1}$, a contradiction.

Since the restriction of a hyperplane $H$ to a plane $\pi$ is a line (if $\pi\nsubseteq H$) or the sum of the lines of a pencil (if $\pi \subseteq H$), it follows that the restriction of the codeword $c$ to a plane $\pi$ is a codeword in the code $C(\pi)$ of points and lines in $\pi$. 

In all planes $\pi$ through $L$, $supp(c)$ has at least two points and $(c,l)=0$ for all lines $l$ in $\pi$, so the restriction of $c$ to $\pi$ lies in $ C(\pi)\cap C(\pi)^\bot$, which has minimum weight $2q$ (see \cite{AK}).

This implies that $supp(c)$ has at least $\theta_{n-2}(2q-2)+2$ points which is equal to $2q^{n-1}$, a contradiction, so the minimum weight of $C\cap C^\bot$ is at least $2q^{n-1}$.

This minimum $2q^{n-1}$ can be obtained when we take the difference of two hyperplanes $H_1$ and $H_2$. This vector has weight $2q^{n-1}$, it is a codeword of $C$ since it is a linear combination of hyperplanes, and it belongs to $C^\bot$ since $(H_1-H_2,H)$=$(H_1,H)-(H_2,H)=0$ for all hyperplanes $H$.
\end{proof}
\begin{remark} Proposition 2 of \cite{Bagchi} yields the same statement for $q$ prime. Moreover, for $q$ prime, every codeword of weight $2q^{n-1}$ in $C\cap C^{\perp}$ is a scalar multiple of the difference of two hyperplanes of $PG(n,q)$.
\end{remark}
\begin{lemma}\label{lem6}
$$C\cap C^\bot=\left\langle H_1-H_2 \vert\vert H_1, H_2 \mbox{ distinct hyperplanes of }PG(n,q) \right \rangle.$$
\end{lemma}
\begin{proof}
Put $A=\left\langle H_1-H_2 \vert\vert H_1, H_2 \mbox{ distinct hyperplanes of }PG(n,q) \right \rangle$. Clearly $A\subseteq C\cap C^\bot$, since $(H,v)=(H,H_i)-(H,H_j)=0$, for every hyperplane $H$ of $PG(n,q)$, and for every $v=H_i-H_j\in A$. Moreover, since $\left \langle A \cup \lbrace H_k \rbrace \right \rangle$ contains each hyperplane, it follows that $\dim(C)-1\leq \dim(A)\leq \dim(C\cap C^\bot)$. The lemma now follows easily, since $C\cap C^\bot$ is not equal to $C$, as a hyperplane is not orthogonal to itself.
\end{proof}

Before we can link codewords of small weight to blocking sets, we need to prove that a small blocking set can be uniquely reduced to a minimal blocking set.

A \emph{blocking set} (with respect to lines) of $PG(n,q)$ is a set $B$ of points such that every line contains at least one point of $B$. 
A blocking set is called \emph{minimal} if no proper subset of it is a blocking set.  A point of a blocking set $B$ is called {\em essential} if it lies on a tangent line to $B$. It is easy to see that a blocking set is minimal if all its points are essential. A blocking set is called {\em trivial} when it contains a hyperplane.

\begin{lemma}{\cite[Lemma 2.11]{nieuw}}\label{lemma1}
Let $B$ be a blocking set in $PG(h+1,q)$ with respect to lines. If $\vert B \vert=2q^h+q^{h-1}+\cdots+q-s$, then there are at least $s+1$ tangent lines through each essential point of $B$.
\end{lemma}








\begin{corollary} \label{cor1}
Every blocking set $B$ w.r.t. lines in $PG(n,q)$, of size smaller than $q^{n-1}+\theta_{n-1}$, can be uniquely reduced to a  minimal blocking set $B'$.
\end{corollary}
\begin{proof}
Suppose that $\vert B \vert=2q^{n-1}+q^{n-2}+\cdots+q-s$, and let $B'$ be a minimal blocking set contained in $B$, with $\vert B'\vert =2q^{n-1}+q^{n-2}+\cdots+q-s'$. A point in $B\backslash B'$ lies on zero tangent lines to $B$. By Lemma \ref{lemma1}, a point $P_1$ of $B'$ lies on at least $s'+1$ tangent lines to $B'$. There are $s'-s$ points in $B\backslash B'$, so $P_1$ lies on at least $s'+1-(s'-s)$ tangent lines to $B$. Since $s\geq 0$, $P_1$ lies on at least one tangent line to $B$. It follows that $B'$ is the set of points of $B$, which lie on at least one tangent line to $B$, and hence, is uniquely determined.
\end{proof}
We are now ready to link codewords of small weight to blocking sets.
\begin{lemma}
\label{lem6b} A codeword $c$ of $C(PG(n,q))$, with weight $wt(c)$ smaller than $2q^{n-1}$, defines a minimal blocking set w.r.t. lines of $PG(n,q)$. Moreover, $c$ is a codeword taking only values from $\lbrace 0,a \rbrace$, for some $a \in \mathbb{F}_p^\star$, and $supp(c)$ intersects every line in $1\pmod{p}$ points.
\end{lemma}
\begin{proof}

Take a codeword $c$ with weight $wt(c)<2q^{n-1}$, then according to Lemmas \ref{lem4}, \ref{lem5} and Theorem \ref{th11}, $(c,l)=a\neq 0$ for every line $l$. So $supp(c)$ defines a blocking set $B$ w.r.t. lines of $PG(n,q)$. We now show that this blocking set is minimal.
Suppose that every line contains at least two points of the blocking set. Counting the points of $B$ on all lines through a point not in $B$ yields
$$\vert B \vert \geq 2\theta_{n-1},$$
a contradiction. So there is a point $R\in B$ lying on at least one tangent line $l$ to $B$. This implies that $(c,l)=c_R=a\neq$ 0. Since $(c,m)=a$ for all lines $m$ (Lemma \ref{lem4}), we may conclude that for every necessary point $R$ of the blocking set $B$ defined by $c$, $c_R$ equals $a \neq 0$.

By way of contradiction, suppose that $c$ defines a non-minimal blocking set, and consider a point $P_1$ that is not necessary. If all $\theta_{n-1}$ lines through $P_1$ contain at least two extra points of $B$, then $\vert B \vert \geq 2\theta_{n-1}+1>2q^{n-1}$, a contradiction. So there is a line $P_1P_2$ which has only $P_1$ and $P_2$ in common with $B$. Since $B$ can be uniquely reduced to a minimal blocking set, see Corollary \ref{cor1}, the point $P_2$ is necessary, which implies that $c_{P_2}=a$. But $a=(c,P_1P_2)=c_{P_1}+c_{P_2}=a+c_{P_1}$, which implies that $c_{P_1}=0$, contradicting $P_1 \in B$. This implies that $B$ is minimal.

Since $(c,m)=a$ for all lines $m$, and $c_P=a$ for all points $P\in supp(c)$, it follows that $supp(c)$ intersects every line in  $1\pmod{p}$ points.
\end{proof}

We give another proof for the following theorem proven in \cite[Proposition 5.7.3]{AK}, by using Lemma \ref{lem6b}.
\begin{corollary} The minimum weight codewords of $C$ are the scalar multiples of the incidence vectors of the hyperplanes of $PG(n,q)$.
\end{corollary}
\begin{proof} According to Lemma \ref{lem6b}, a codeword of weight smaller than $2q^{n-1}$ is a scalar multiple of the incidence vector of a minimal blocking set with respect to lines. A result of Bose and Burton \cite{Bose} shows that the minimum size of a blocking set with respect to lines in $PG(n,q)$ is equal to $\theta_{n-1}$, and that this minimum is reached if and only if the blocking set is a hyperplane.
\end{proof}

The following lemmas are extensions of Lemmas 6.6.1 and 6.6.2 of Assmus and Key \cite{AK}. 
\begin{lemma} \label{lem7}A vector $v$ of $V(\theta_n,p)$ taking only values from $\lbrace 0,a \rbrace$, $a\in \mathbb{F}_p^\star$, is contained in $(C\cap C^\bot)^\bot$ if and only if $|supp(v) \cap H|\pmod{p}$ is independent of the hyperplane $H$ of $PG(n,q)$.
\end{lemma}
\begin{proof}Let $v$ be a vector in $(C\cap C^\bot)^\bot$, then $(v, H_1-H_2)=0$ since $C\cap C^\bot$ is generated by the differences of the hyperplanes (Lemma \ref{lem6}). We see that $(v,H)= a |supp(v) \cap H|\pmod{p}$ is independent of the choice of the hyperplane $H$ and so is $|supp(v) \cap H|\pmod{p}$. 

Conversely, if $|supp(v) \cap H|$ is constant $\pmod{p}$, then $(v,H) = a \vert supp(v) \cap H\vert\pmod{p}$. This implies that $(v, H_1-H_2)=0$ for all hyperplanes $H_1,H_2$, and hence $v\in(C\cap C^\bot)^\bot$. 
\end{proof}

\begin{lemma} \label{lem8} Let $c$,$v$ be two vectors taking only values from $\lbrace 0,a \rbrace$, $a \in \mathbb{F}_p^\star$, with $c$ $ \in C$, $v \in (C\cap C^\bot)^\bot$. If $\vert supp(c) \cap H\vert\equiv \vert supp(v) \cap H\vert\pmod{p}$ for every hyperplane $H$, then $\vert supp(c) \cap supp(v)\vert \equiv \vert supp(c)\vert\pmod{p}$.
\end{lemma}

\begin{proof} According to Lemma \ref{lem6}, $(c,H_1-H_2)=0$ for all hyperplanes $H_1,H_2$, since $c$ $\in$ $C$.Hence, $\vert supp(c) \cap H \vert$ $\pmod{p}$ is independent of the hyperplane $H$. Since $(c-v,H)=(c,H)-(v,H)\equiv a\vert supp(c) \cap H\vert-a\vert supp(v) \cap H\vert \equiv 0\pmod{p}$, for every hyperplane $H$, it follows that $c-v \in C^{\bot}$, and hence $(c-v,c)\equiv a^2\vert supp(c)\vert-a^2\vert supp(c) \cap supp(v)\vert\equiv 0\pmod{p}$. This yields that $\vert supp(c)\vert \equiv \vert supp(c)\cap supp(v) \vert\pmod{p}$. \end{proof}

As mentioned in the introduction, we will eliminate all so-called non-trivial {\it linear} blocking sets as the support of a codeword of $C$ of small weight. In order to define a linear blocking set, we introduce the notion of a Desarguesian spread.

By what is sometimes called "field reduction", the points of $PG(n,q)$, $q=p^h$, $p$ prime, correspond to $(h-1)$-dimensional subspaces of $PG((n+1)h-1,p)$, since a point of $PG(n,q)$ is a $1$-dimensional vector space over ${\mathbb F}_q$, and hence an $h$-dimensional vector space over ${\mathbb F}_p$. In this way, we obtain a partition ${\mathcal D}$ of the point set of $PG((n+1)h-1,p)$ by $(h-1)$-dimensional subspaces. In general, a partition of the point set of a projective space by subspaces of a given dimension $k$ is called a {\it spread}, or if we want to specify the dimension, a {\it $k$-spread}. The spread we have obtained here is called a {\it Desarguesian spread}. Note that the Desarguesian spread satisfies the property that each subspace spanned by two spread elements is again partitioned by spread elements. In fact, it can be shown, see \cite{L1}, that if the dimension of the ambient space is larger than twice the dimension of a spread element plus one (i.e. $n\geq 2$), then this property characterises a Desarguesian spread.

\begin{definition} Let $U$ be a subset of $PG((n+1)h-1,p)$ and let $\mathcal{D}$ be a Desarguesian $(h-1)$-spread of $PG((n+1)h-1,p)$, then $\mathcal{B}(U)=\lbrace R \in \mathcal{D}\vert \vert U\cap R \neq \emptyset \rbrace$.\end{definition}

Analogously to the correspondence between the points of $PG(n,q)$ and the elements of a Desarguesian spread $\mathcal D$ in $PG((n+1)h-1,p)$, we obtain the correspondence between the lines of $PG(n,q)$ and the $(2h-1)$-dimensional subspaces of $PG((n+1)h-1,p)$ spanned by two elements of $\mathcal D$. With this in mind, it is clear that any $(nh-h)$-dimensional subspace $U$ of $PG(nh+h-1,p)$ defines a blocking set ${\mathcal B}(U)$ w.r.t. lines in $PG(n,q)$. A blocking set constructed in this way is called a {\it linear} blocking set.  Linear blocking sets were first introduced by Lunardon \cite{L1}, although there a different approach was used.
For more on the approach explained here, we refer to \cite{lavrauw2001}. 

\begin{remark} When working with this representation, we assume that $h>1$. We deal with the case $h=1$ in Corollary \ref{gevolg}.
\end{remark}

\begin{lemma}\label{lem9}
If  $U$ is a subspace of $PG((n+1)h-1,q)$, then $\vert \mathcal{B}(U) \vert =1\pmod{q}$.
\end{lemma}
\begin{proof}
Suppose that $U$ is a subspace of $PG((n+1)h-1,q)$ of dimension $r$ and let $X_i$ be the number of spread elements intersecting $U$ in a subspace of dimension $i$. Each point of $U$ lies in a unique spread element, so

$$\sum_{i=0}^r X_i\theta_i=\theta_r \Leftrightarrow$$
$$\sum_{i=0}^r X_iq^{i+1}-\sum_{i=0}^r X_i=q^{r+1}-1 \Leftrightarrow$$
$$q(\sum_{i=0}^rX_iq^i-q^r)=\sum_{i=0}^rX_i-1.$$

The left hand side is divisible by $q$, so $\sum_{i=0}^r X_i=\vert \mathcal{B}(U) \vert =1\pmod{q}$.
\end{proof}

We put $N=h(n-1)$ throughout the following proofs. We call  the linear blocking set $B$ of $PG(n,q)$ defined by $\mathcal{B}(U_N)$, where $U_N$ is an $N$-dimensional subspace of $PG(h(n+1)-1,p)$, a {\em small} linear blocking set. Such a small linear blocking set is always minimal. Our goal is to exclude the incidence vectors of small linear blocking sets as codewords of $C(PG(n,q))$.

\begin{lemma}\label{N1} Let $U_N$ be an $N$-dimensional subspace of $PG(h(n+1)-1,p)$. Then the number of spread elements of $\mathcal{B}(U_N)$ intersecting $U_N$ in exactly one point is at least $p^{hn-h}-p^{hn-h-2}-p^{hn-h-3}-\cdots-p^{hn-2h+1}-p^{hn-2h-2}-\cdots-p^{hn-3h+1}-p^{hn-3h-2}-\cdots-p^{h+1}-p^{h-2}-\cdots-p$ .\end{lemma}
\begin{proof} The set $\mathcal{B}(U_N)$ defines a blocking set $B$ in $PG(n,q)$, $q=p^h$, $p$ prime, $h\geq 1$, w.r.t. the lines. So $\vert \mathcal{B}(U_N) \vert =\vert B \vert\geq$ $(q^n-1)/(q-1)$ $=(p^{hn}-1)/(p^h-1)$ by Bose and Burton \cite{Bose}. Suppose that there are exactly $x$ spread elements of $\mathcal{B}(U_N)$ intersecting $U_N$ in one point, then
$$\frac{p^{hn}-1}{p^h-1}\leq \vert B \vert \leq \frac{\vert U_N\vert-x}{p+1}+x.$$
Using that $\vert U_N \vert=(p^{N+1}-1)/(p-1)$ yields that $x\geq p^{hn-h}-p^{hn-h-2}-p^{hn-h-3}-\cdots-p^{hn-2h+1}-p^{hn-2h-2}-\cdots-p^{hn-3h+1}-p^{hn-3h-2}-\cdots-p^{h+1}-p^{h-2}-\cdots-p$.
\end{proof}

\begin{remark} \label{rem1} It follows from Lemma \ref{N1} that the number of spread elements of $\mathcal{B}(U_N)$ intersecting $U_N$ in exactly one point is at least $p^{N}-p^{N-1}+1$. We will use this weaker bound.
\end{remark}

\begin{lemma} \label{N2} If there are $p^{N}-p^{N-1}+1$ points $R_i$, $i=1,\ldots,p^N-p^{N-1}+1$, of a minimal blocking set $B$ in $PG(n,q)$, $q=p^h>2$, for which it holds that every line through $R_i$ is either a tangent line to $B$ or is entirely contained in $B$, then $B$ is a hyperplane of $PG(n,q)$.
\end{lemma}
\begin{proof}
It is easy to see that a plane through a line $R_iR_j$, $i\neq j$, is either completely contained in $B$, or intersects $B$ only in the line $R_iR_j$.
There are at least $(p^{N}-p^{N-1})/p^h$ different lines $R_1R_i$, $i\neq 1$.

We prove that if $B \supset \pi_m$, $B \neq \pi_m$, for some $m$-dimensional space $\pi_m$ through $R_1$, then $B \supseteq \pi_{m+1}$ for some $(m+1)$-dimensional space through $\pi_m$ for all $m< n-1$. 

If $B\supset \pi_m$, then there are still $(p^{N}-p^{N-1})/p^h-(p^{hm}-1)/(p^h-1)$ lines $R_1R_j$ through $R_1$, but not in $\pi_m$,  such that every plane through it intersects $B$ in this line or lies completely in $B$. We can choose such a line $R_1R_j$ if $m<n-1$ and $p^h>2$. Then the space $\langle R_1R_j,\pi_m\rangle$ is clearly contained in $B$.
By induction, we find a hyperplane $\pi$ contained in $B$. Since $B$ is minimal, $B=\pi$.
\end{proof}

\begin{remark} It follows from the proof of Lemma \ref{N2} that it is sufficient to find $n-1$ linearly independent points $R_i$ such that every line through $R_i$ is either a tangent line to $B$ or is entirely contained in $B$, to prove that $B$ is a hyperplane. Moreover, this bound is tight. If there are only $n-2$ linearly independent points for which this condition holds, we have the example of a Baer cone, i.e. let $B$ be the set of all lines connecting a point of a Baer subplane $\pi=PG(2,\sqrt{q})$ to the points of an $(n-3)$-dimensional subspace of $PG(n,q)$, skew to $\pi$.
\end{remark}

%

\begin{lemma}\label{lem10}
Let $U_{N-1}$ be a fixed $(N-1)$-dimensional space in $PG(h(n+1)-1,p)$ and let $U_N$ be an arbitrary $N$-dimensional space containing $U_{N-1}$, $N>2$. Then $\mathcal{B}(U_N)$ is uniquely determined by $U_{N-1}$ and two elements $R_1$, $R_2$ $\in$ $\mathcal{B}(U_N) \backslash \mathcal{B}(U_{N-1})$.
\end{lemma}
\begin{proof} We may assume that $\mathcal{B}(U_{N-1}) \neq \mathcal{B}(U_N)$, since the theorem is obvious if $\mathcal{B}(U_{N-1}) = \mathcal{B}(U_N)$. 

Suppose that $R_1,R_2 \in \mathcal{B}(U_N)\backslash \mathcal{B}(U_{N-1}), R_1 \neq R_2$. If $R_3 \in \mathcal{B}(U_N) \backslash \mathcal{B}(U_{N-1})$, $R_2 \neq R_3 \neq R_1$, then we claim that $R_3$ can be constructed only using elements of $\mathcal{B}(U_{N-1}) \cup \lbrace R_1,R_2\rbrace$.
Clearly, $R_i$ intersects $U_N$ in a point $P_i$ since $R_1, R_2$ and $R_3$ are elements of $\mathcal{B}(U_N) \backslash \mathcal{B}(U_{N-1})$. So $\langle P_1, P_3 \rangle$ intersects $U_{N-1}$ in a point $P_4$ which lies on a unique spread element $R_4$. Similarly, the spread element through $\langle P_2,P_3 \rangle \cap U_{N-1}$ is called $R_5$.

Case 1: $P_3 \notin P_1P_2$. The spaces $\langle R_1,R_4 \rangle$ and $\langle R_2,R_5\rangle$ are spanned by two elements of a Desarguesian spread, so they intersect in a spread element. The intersection of $\langle R_1,R_4 \rangle$ with $\langle R_2,R_5\rangle$ certainly contains $R_3$. We can conclude that $R_3=\langle R_1,R_4\rangle \cap \langle R_2,R_5 \rangle$.

Case 2: $P_3 \in P_1P_2$.
Take a spread element $R_6 \in \mathcal{B}(U_N)$ already constructed in Case 1. We can switch $R_6$ with $R_2$. Then $R_3 \notin \langle R_1,R_6 \rangle$. So we can copy the proof of Case 1 to determine $R_3$ from $R_1, R_6$ and $U_{N-1}$. But $R_6$ was determined by $R_1, R_2$ and $U_{N-1}$, hence so is $R_3$.
\end{proof}

\begin{theorem}\label{the5}
For every small linear blocking set $B$ w.r.t. lines, not defining a hyperplane in $PG(n,p^h)$, there exists a small linear blocking set $B'$ intersecting $B$ in $2\pmod{p}$ points.
\end{theorem}
\begin{proof}
As we have seen before, a small linear blocking set $B$ in $PG(n,p^h)$ corresponds  to an $N$-dimensional space $U_N$ in $PG(h(n+1)-1,p)$. We will construct a subspace $U_{N}'$ that defines a second blocking set $B'$ intersecting $B$ in $2 \pmod{p}$ points.



There is a spread element $R'$, lying in a $(2h-1)$-dimensional space spanned by two spread elements $R_1$ and $R_2$, $R_1,R_2\in \mathcal{B}(U_N)$, where $R_1\cap U_N$  is a point, such that $R'$ does not intersect $U_N$. Suppose that for every $R'_1$ and $R'_2$ in $\mathcal{B}(U_N)$, where $R_1' \cap U_N$ is a  point, each spread element in $\langle R'_1,R'_2\rangle$ intersects $U_N$. Then $\mathcal{B}(U_N)$ defines a set $B$ of points in $PG(n,q)$ such that every line through $R'_1$ is tangent to $B$ in the point $R'_1$ or is entirely contained in $B$. But Remark \ref{rem1} and Lemma \ref{N2} then imply that $B$ is a hyperplane, a contradiction.

Choose an $(N-1)$-dimensional space $U_{N-1} \subset U_N$, such that $R_2\in \mathcal{B}(U_{N-1})$ and $R_1\notin \mathcal{B}(U_{N-1})$.

The elements $R_1,R_2,R'$ define an $(h-1)$-regulus. Take a transversal line $m$ to this $(h-1)$-regulus intersecting $U_{N-1}$ in a point of $U_{N-1} \cap R_2$. Then $\langle m,U_{N-1} \rangle$ is an $N$-dimensional space $U_N'$, defining a blocking set $B'$ of $PG(n,q)$.

Now $\mathcal{B}(U_N)$ and $ \mathcal{B}(U_N')$ have $\mathcal{B}(U_{N-1})$ and $R_1$ in common. So $B$ and $B'$ have at least $(1\mod p) +1$ points in common (see Lemma \ref{lem9}). 

If $\mathcal{B}(U_N)\cap \mathcal{B}(U'_N)$ contains another spread element $R_3 \notin \mathcal{B}(U_{N-1})$, $R_3 \neq R_1$, then Lemma \ref{lem10} implies that $\mathcal{B}(U_N)=\mathcal{B}(U'_N)$, contradicting $R'\in \mathcal{B}(U'_N)\backslash \mathcal{B}(U_N)$. It follows that the blocking sets $B$ and $B'$ of $PG(n,q)$ corresponding to $U_N$ and $U'_N$ intersect in $2 \pmod{p}$ points. 
\end{proof}

Using this result, we exclude in Theorem \ref{th4} all small non-trivial linear blocking sets as codewords.
\begin{theorem}\label{th4}
Let $v$ be the incidence vector of a small non-trivial linear blocking set of points w.r.t. lines of $PG(n,q)$. Then $v\notin C$. 
\end{theorem}
\begin{proof}
We know that $\vert supp(v) \vert\equiv 1\pmod{p}$. We know from Theorem \ref{the5}  that there exists a linear minimal blocking set $w$ such that $\vert supp(v) \cap supp(w) \vert\equiv 2\pmod{p}$. Since $\vert supp(w) \cap H\vert \equiv 1\pmod{p}$ for every hyperplane $H$ (see Lemma \ref{lem9}), it follows that $w \in (C \cap C^\bot)^\bot$ (Lemma \ref{lem7}). Suppose that $v \in C$, then Lemma \ref{lem8} implies that $\vert supp(v) \cap supp(w)\vert \equiv \vert supp(v) \vert$ $\equiv 1\pmod{p}$, a contradiction.
\end{proof}
Together with Lemma \ref{lem6b}, Theorem \ref{th4} gives the following  corollary.
\begin{corollary} \label{gev1}
The only possible codewords $c$ of $C$ of weight in $]\theta_{n-1},2q^{n-1}[$ are the scalar multiples of non-linear minimal blocking sets, intersecting every line in $1\pmod{p}$ points.
\end{corollary}

\begin{remark}
Amongst many of the leading researchers dealing with blocking sets, it is believed (and conjectured, see \cite{sziklai}) that all small minimal blocking sets are linear. If that conjecture is true, then Corollary \ref{gev1} eliminates all possible codewords of weight in $]\theta_{n-1},2q^{n-1}[$. The cases in which the conjecture is proven (and relevant here) are mentioned below.
\end{remark}

In some cases, we can exclude non-linear blocking sets intersecting every line in $1\pmod{p}$ points. 
\begin{lemma}\label{lem:1}
The only minimal blocking set $B$ in $PG(n,p)$, with $p$ prime, such that every line contains $1\pmod{p}$ points of $B$, is a hyperplane.
\end{lemma}
\begin{proof}Let $B$ be a blocking set in $PG(n,p)$ such that every line intersects $B$ in $1\pmod{p}$ points. If $B \supset PG(m,p)$, $B \neq PG(m,p)$, for some $m$, then $B \supseteq PG(m+1,p)$ since we can connect a point $R'$ in $B \setminus PG(m,p)$ to all points of $PG(m,p)$. All these lines have to lie in $B$, so $PG(m+1,p)=\langle R',PG(m,p)\rangle\subset B$. There is always a line skew to $PG(m,p)$, with $m<n-1$, so we can always find a point $R'\in B \setminus PG(m,p)$ for $m<n-1$. This implies that the only possibility for a minimal blocking set $B$ such that every line has $1\pmod{p}$ points of $B$, is a hyperplane $PG(n-1,p)$.
\end{proof}
The next corollary, following from Lemma \ref{lem:1},  extends the result of   Chouinard (Theorem \ref{ch} (1)) to general dimension.
\begin{corollary} \label{gevolg}
There are no codewords $c$, with $\theta_{n-1} <wt(c)<2p^{n-1}$, in $C(PG(n,p))$, $p$ prime.
\end{corollary}

We  turn our attention to minimal blocking sets $B$, with $|B|\in ]\theta_{n-1},2q^{n-1}[$, in $PG(n,q)$, $q=p^h$,  $p$ prime, $h\geq 1$, such that every line contains $1\pmod{p}$ points of $B$. Let $e$ be the maximal integer for which $B$ intersects every line in $1\pmod{p^e}$ points. Then results of Sziklai prove that $e$ is a divisor of $h$ \cite{sziklai}.

In \cite[Corollary 5.2]{SF:07}, it is proven that 
\[|B|\geq q^{n-1}+\frac{q^{n-1}}{p^e+1}-1.\]

We now derive the upper bound on $|B|$, based on \cite[Theorem 5.3]{SF:07}.

\begin{theorem} \label{th:2} Let $B$ be a minimal blocking set w.r.t. the lines of $PG(n,q)$, $q=p^h$, $p$ prime, $h\geq 1$, intersecting every line in $1\pmod{p^e}$ points, with $e$ the maximal integer for which this is true, and assume that $|B| \in ]\theta_{n-1},2q^{n-1}[$ and that $p^e>3$.

Then \[|B| \leq q^{n-1}+\frac{2q^{n-1}}{p^e}.\]
\end{theorem}

\begin{proof} Let $E=p^e$. Let $\tau_{1+iE}$ be the number of lines intersecting
$B$ in $1+iE$ points.
We count the number of lines, the number of  pairs
$(R,l)$, with $R \in B$ and with $l$ a line through $R$, and
the number of triples $(R,R',l)$, with $R$ and $R'$ distinct points
of $B$ and $l$ a line passing through $R$ and $R'$.

Then the following formulas are valid:

\begin{eqnarray}
 \sum_{i \geq 0} \tau_{1+iE}&=&\frac{(q^{n+1}-1)(q^n-1)}{(q^2-1)(q-1)},\\
\sum_{i \geq 0} (1+iE)\tau_{1+iE} & = & |B|
\left(\frac{q^n-1}{q-1}\right),\\
\sum_{i \geq 0} (1+iE)(1+iE-1)\tau_{1+iE} &=&  |B|(|B|-1).
\end{eqnarray}

Then $\sum_{i\geq 0} i(i-1) E^2 \tau_{1+iE} \geq 0$
implies that

\[|B|(|B|-1) -(1+E)|B|\left(\frac{q^n-1}{q-1}\right)
+(1+E)\frac{(q^{n+1}-1)(q^n-1)}{(q^2-1)(q-1)} \geq 0.\]

Under the condition $3<E$ and $\vert B \vert \in ]\theta_{n-1},2q^{n-1}[$, this implies that \[|B| \leq 
q^{n-1}+\frac{2q^{n-1}}{E}.\]

\end{proof} 

To exclude codewords in the code of $PG(n,p^2)$, with $p$ a prime, we can use the following theorem of Weiner which implies that every small minimal blocking set in $PG(n,p^2)$ is linear.

\begin{theorem}\label{Weiner} \cite{Weiner} A non-trivial minimal blocking set of $PG(n,p^2)$, $p>11$, $p$ prime, with respect to $k$-spaces and of size less than $3(p^{2(n-k)}+1)/2$ is a $(t,2((n-k)-t-1))$-Baer cone with as vertex a $t$-space and as base a $2((n-k)-t-1)$-dimensional Baer subgeometry, where $\max \lbrace -1,n-2k-1 \rbrace \leq t < n-k-1$.
\end{theorem}

Theorem \ref{Weiner}, together with Theorem \ref{th:2}, yields the following corollary.
\begin{corollary} \label{kw}
There are no codewords $c$, with $wt(c) \in ]\theta_{n-1},2q^{n-1}[$, in $C(PG(n,q))$, $q=p^2$, $p>11$, $p$ prime.
\end{corollary}

For general $q=p^h$, $p$ prime, $h\geq 3$, Theorem \ref{th:2} implies that the weights of possible codewords $c$ in $C$, with $wt(c) \in ]\theta_{n-1},2q^{n-1}[$, corresponding to non-linear blocking sets intersecting every line in $1\pmod{p^e}$ points, with $e$ the maximal integer for which this is true,  must belong to certain small intervals.

In particular, we exclude all the codewords with weight in $[3q^{n-1}/2,2q^{n-1}[$; in this way, excluding half of the interval $]\theta_{n-1},2q^{n-1}[$.

\begin{corollary}
There are no codewords $c$ in $C(PG(n,q))$, $q=p^h$, $p$ prime, $p>3$, $h\geq3$, with weight in $[3q^{n-1}/2,2q^{n-1}[$.
\end{corollary}

\section{Minimum weight codewords in the dual code generated by the incidence matrix of points and hyperplanes of $PG(n,q)$}
In this section, we consider codewords $c$ $\in$ $C(PG(n,q))^\bot$, $q=p^h$, $p$ prime, $h\geq 1$, with $C(PG(n,q))$ the $p$-ary linear  code generated by the incidence matrix of points and hyperplanes in $PG(n,q)$, $q=p^h$, $p$ prime, $h\geq 1$. This means that $(c,H)=0$ for all hyperplanes $H$ of $PG(n,q)$, since a codeword in $C^\bot$ is orthogonal to all the rows of the generator matrix of $C$.

For every hyperplane $H$, $$\sum_{P\in supp(c)\cap H}c_{P} = 0.$$

Denote the minimum distance of a linear code $C$ by $d(C)$. Note that $d(C^\perp)\leq 2q$ since the difference of the incidence vectors of two intersecting lines is  a codeword of $C^\perp$. 

\begin{lemma}\label{min}
For each $n\geq 3$, the following holds:
$$d(C(PG(n,q))^\bot) \geq d(C(PG(n-1,q))^\bot)\geq \cdots \geq d(C(PG(2,q))^\bot).$$
\end{lemma}
\begin{proof}
Let $c$ be a codeword of $C(PG(n,q))^\bot$ of minimum weight, and let $R$ be a point of $PG(n,q)\backslash supp(c)$, with $R$ on a tangent line to $supp(c)$, and let $H$ be a hyperplane of $PG(n,q)$ not containing $R$. For each point $P\in H$, define $c'_P=\sum c_{P_i}$, with $P_i$ the points of $supp(c)$ on the line $\langle R,P\rangle$, and let $c'$ denote the vector with coordinates $c'_P$, $P\in H $. Note that $c'\neq 0$, since $R$ lies on a tangent line to $supp(c)$.

Then it easily follows that $c'\in C(PG(n-1,q))^\bot$, and $supp(c')$ is contained in the projection of $supp(c)$ from the point $R$ onto the hyperplane $H=PG(n-1,q)$. Clearly, $\vert supp(c')\vert \leq \vert supp(c) \vert$. 

Using this relation on a codeword $c$ of minimum weight  yields that 
$d(C(PG(n-1,q))^\bot)\leq d(C(PG(n,q))^\bot)$. Continuing this process proves the statement.
\end{proof}

\begin{remark} We call the vector $c'$ defined in the proof of Lemma \ref{min}, the projection of $c$.
\end{remark}
\begin{theorem}\label{st1} For each $n\geq3$, $d(C(PG(n,q))^\bot)=d(C(PG(2,q))^\bot)$.
\end{theorem}
\begin{proof}
Embed $\pi=PG(2,q)$ in $PG(n,q)$, $n>2$, and extend each codeword $c$ of $C(\pi)^\bot$ to a vector $c^{(n)}$ of $V(\theta_n,p)$ by putting a zero at each point $P\in PG(n,q)\backslash \pi$. Since the all one vector of $V(\theta_2,p)$ is a codeword of $C(PG(2,q))$, it follows that $\sum_{P\in \pi} c_P^{(n)}=0$ for each $c^{(n)}$. 

This implies that $(c^{(n)},H)=0$, for each hyperplane $H$ of $PG(n,q)$ which contains $\pi$. If a hyperplane $H$ of $PG(n,q)$ does not contain $\pi$, then $(c^{(n)},H)=(c,H\cap \pi)=0$, since $(c,l)=0$, for each line $l$ of $\pi$.

 It follows that $c^{(n)}$ is a codeword of $C(PG(n,q))^\bot$ of weight equal to the weight of $c$, which implies that $d(C(PG(n,q))^\bot)\leq d(C(PG(2,q))^\bot)$. Regarding Lemma \ref{min}, this yields that $d(C(PG(n,q))^\bot)=d(C(PG(2,q))^\bot)$.
\end{proof}

\begin{lemma} \label{lem1}
Let $B$ be a set in $PG(n,q)$, with $\dim\langle B\rangle \geq 3$, such that any point $R$ in $PG(n,q)\backslash B$ that lies on at least one secant line to $B$, does not lie on tangent lines to $B$. Then $\vert B \vert \geq 3q$.
\end{lemma}

\begin{proof}
We first prove the following result. When we take two secants $l_1, l_2$ through $R$, then the plane $\langle l_1,l_2 \rangle$ contains at least $q+ \max \lbrace a_1,a_2 \rbrace$ points of $B$, where $a_i=\vert l_i \cap B \vert$. 
Take a point $S\in B$ on $l_1\backslash l_2$. Then every line in $\langle l_1,l_2 \rangle$ through $S$ must be a secant line to $B$; else if it lies on a tangent line $l$, $l\cap l_2$ is a point not in $B$ lying on a tangent line and a secant line to $B$, which is a contradiction. So $\vert B \cap \langle l_1,l_2 \rangle \vert \geq q+a_1$, and similarly, $\vert B \cap \langle l_1,l_2 \rangle\vert \geq q+a_2$.

Now $R$ lies on at least three non-coplanar secants to $B$, since $\dim \langle B \rangle \geq 3$. Now $$\vert \langle l_1,l_2 \rangle \cap B \vert \geq q+\max\lbrace a_1,a_2 \rbrace,$$
$$\vert \langle l_1,l_3 \rangle \cap B \vert \geq q+\max\lbrace a_1,a_3 \rbrace,$$
$$\vert \langle l_2,l_3 \rangle \cap B \vert \geq q+\max\lbrace a_2,a_3 \rbrace,$$
with $a_i =|l_i\cap B|$.

So $\vert B \vert \geq (q+\max\lbrace a_1,a_2 \rbrace)+ (q+\max\lbrace a_1,a_3 \rbrace)+(q+\max\lbrace a_2, a_3 \rbrace)$ $-(a_1+a_2+a_3)$, because we counted the points lying on $l_i\cap B$ twice. It follows that $\vert B \vert \geq 3q$.
\end{proof}
\begin{theorem} \label{th3}
Let $c$ be a codeword of $C(PG(n,q))^\bot$, $n\geq 3$, of minimal weight, then  $supp(c)$ is contained in a plane of $PG(n,q)$.
\end{theorem}
\begin{proof}The difference of two intersecting lines clearly belongs to the dual code and has weight $2q$, so we may assume that $wt(c) \leq 2q$.

Assume that $\dim\langle supp(c) \rangle \geq 3$; using Lemma \ref{lem1}, we find a point $R$ lying on a tangent line to $supp(c)$ and lying on at least one secant line to $supp(c)$. It follows from Theorem \ref{st1} that $wt(c)=d(C(PG(n,q))^\bot)=d(C(PG(n-1,q))^\bot)=d(C(PG(2,q))^\bot).$
 
Since $R$ lies on at least one secant line and at least one tangent line to $supp(c)$, the projection $c'$, of $c$ from $R$,  has weight smaller than $wt(c)$.

But then $c'$ is a non-zero codeword of $C(PG(n-1,q))^\bot$ satisfying $0<wt(c')\leq wt(c)-1<d(C(PG(n-1,q))^\bot)$, a contradiction.
\end{proof}

In Theorem \ref{th3}, we reduced the problem of finding the minimum weight of the dual of the code generated by points and hyperplanes in $PG(n,q)$ to finding the minimum weight of the dual of the code generated by points and lines in $PG(2,q)$. This means that we can use the known results about this latter code.

From \cite[Theorem 6.4.2]{AK}, we get the following bound on the minimum weight $d$ of $C(PG(2,q))^\bot$, with $q=p^h, p$ prime, $h\geq 1$: 
$$q+p\leq d \leq 2q,$$
with equality at the lower bound for $p=2$.

 Using this bound, together with Theorem \ref{th3}, yields the following three theorems. 
\begin{theorem}\label{priem}
The minimum weight of $C(PG(n,p))^\bot$, $p$ prime, is equal to $2p$.
\end{theorem}
\begin{theorem}
The minimum weight of $C(PG(n,2^h))^\bot$ is equal to $2^h+2$.
\end{theorem}
\begin{theorem}\label{AK}If $d$ is the minimum weight of $C(PG(n,q))^\bot$, $q=p^h$, $p$ prime, then
$$q+p\leq d \leq 2q.$$
\end{theorem}
We conclude this manuscript by improving on Theorem \ref{AK}. We summarize the improved bounds on the minimum weight of $C(PG(n,q))^\bot$  in  Table 1 at the end of this section.

In Theorem \ref{th11}, it was proven that the minimum weight of $C\cap C^\bot$ is equal to $2q^{n-1}$. We now show that the minimum weight of $C^\bot$ is smaller than $2q$ under certain conditions.
\begin{theorem} \label{st}
Let $B$ be a minimal blocking set in $PG(2,q)$ of size $q+k$, with $k<(q+3)/2$, of R\'edei-type (i.e. there exists a $k$-secant $L$). Then the difference of the incidence vectors of $B$ and $L$ is a codeword of $C(PG(2,q))^\bot$ with weight $2q+1-k$.
\end{theorem}
\begin{proof}
If $k<(q+3)/2$, then $B$ is a small minimal blocking set, hence every line intersects $B$ in $1\pmod{p}$ points (see \cite{SZ}). Let $c_1$ be the incidence vector of $B$ and let $c_2$ be the incidence vector of $L$. Then $(c_1-c_2,m)=(c_1,m)-(c_2,m)=0$ for all lines $m$, hence $c_1-c_2$ is a codeword of $C(PG(2,q))^\bot$, with weight $2q+1-k$.
\end{proof}
We can use this theorem to lower the upper bound on the possible minimum weight of codewords of $C(PG(2,q))^\bot$.
Let $q=p^h$, let $e$ be a divisor of $h$ with $1<e<h$, then we have the following linear blocking set
$$B=\left\lbrace (1,x,x^{p^e})\vert \vert x \in \mathbb{F}_{p^h} \right \rbrace \cup \left\lbrace (0,x,x^{p^e})\vert \vert x \in \mathbb{F}_{p^h},x\neq 0 \right \rbrace.$$

The size of such a blocking set is $q+\frac{q-1}{p^e-1}$. The second part belongs to  a line $L$ which is a $\frac{q-1}{p^e-1}$-secant, so the weight of the codeword arising from the difference of the incidence vectors of $B$ and $L$ is equal to $2q+1-\frac{q-1}{p^e-1}$.

\begin{corollary}
For $q=p^h$, $p$ prime, $h\geq 1$, $d(C(PG(2,q))^\perp) \leq 2q+1-(q-1)/(p-1)$.
\end{corollary}

\begin{remark} In \cite[p. 130]{Bagchi}, the authors write that they have no examples of codewords of $C^\bot$ with weight smaller than $2q$, where $q$ is odd. Theorem \ref{st} provides numerous examples of such codewords for even and odd $q$.
\end{remark}

The following result of Sachar \cite{sachar} states a lower bound on the minimum weight of $C^\bot$.

Let $\Pi$ be a, not necessarily Desarguesian, projective plane of order $n$. Let $C_p(\Pi)$ denote the $p$-ary code of points and lines of $\Pi$, with $p\vert n$.
\begin{theorem}\cite{sachar} \label{th4'}
Let $c$ be a codeword of minimum weight of $C_p(\Pi)^\bot$, and suppose that $p\nmid wt(c)$. If $p=5$, then $wt(c)\geq 4(2n+3)/5$, and if $p> 5$, then $wt(c)\geq (12n+18)/7$.
\end{theorem}

We give a modification of the proof for the second part of Theorem \ref{th4'}, with a small change in the case $p=7$, which has as convenience that the condition $p\nmid wt(c)$ is not necessary.
\begin{remark} Let $c$ be a codeword of $C_p(\Pi)$, $p>2$, with $wt(c)\leq 2n+2$. Since through every point of $supp(c)$, there is a $2$-secant, it is easy to see that the number of distinct non-zero symbols used in $c$ must be even, and that the distinct non-zero symbols occuring in $c$ occur can be partitioned into pairs $\{a,-a\}$.
\end{remark}
In this modification of the proof, we use the following lemma of Sachar.
\begin{lemma} \cite[Proposition 2.2]{sachar}\label{lem2} Suppose that there are $2m$ different non-zero symbols used in the codeword $c\in C_p(\Pi)^\bot$, with $wt(c)\leq 2n+2$. Then $wt(c)\geq n+\frac{2m-1}{2m+1}n+\frac{6m}{2m+1}$.
\end{lemma}

\begin{theorem}\label{th4b}
Let $c$ be a codeword of minimum weight of $C(PG(2,q))^\bot$, $q=p^h$, $p$ prime, $h\geq 1$. If $p=7$, then $wt(c)\geq (12q+6)/7$, and if $p>7$, then $wt(c)\geq (12q+18)/7$.
\end{theorem}
\begin{proof}

 Let $c$ be a codeword of minimum weight of $C^\bot$ and suppose that $wt(c)<(12q+18)/7$. Then it follows from Lemma \ref{lem2} that there are at most four different non-zero symbols used in the codeword $c$. 

Suppose first that there are exactly two non-zero symbols used in $c$, say $1$ and $-1$. Suppose that the symbol $-1$ occurs the least, say $y$ times.
Let $X_S$ be the  number of $2$-secants through a point $S$ of $supp(c)$. Let $R$ be a point of $supp(c)$ for which $c_R=1$. At most $y$ of the lines through $R$ contain a point $R'$ of $supp(c)$ with $c_{R'}=-1$, so at least $q+1-y$ of those lines only contain points $R'$ of $supp(c)$ with $c_{R'}=1$. Since $(c,l)=0$ for all lines $l$, such lines contain $0 \pmod{p}$ points of $supp(c)$. Then
$$wt(c)\geq (q+1-y)(p-1)+y+1.$$

If $wt(c)<(12q+6)/7$, then $y<(6q+3)/7$, and  this implies that 
$$q+1> (q+4)p/7+1;$$
a contradiction if $p=7$. If $wt(c)<(12q+18)/7$, then $y<(6q+9)/7$, and this implies that
$$q \geq (q-2)p/7;$$ a contradiction if $p>7$.

Assume now that there are four non-zero symbols, say $1,-1,a,-a$, in $c$. We can copy the arguments of the proof of Sachar \cite{sachar} to obtain the stated lower bound.
\end{proof}

Using Theorem \ref{th3}, together with Theorem \ref{th4b}, proves that the following result holds.

\begin{theorem}Let $c$ be a codeword of minimum weight of $C(PG(n,q))^\bot$, $q=p^h$, $p$ prime, $h\geq 1$. If $p=7$, then $wt(c)\geq (12q+7)/7$, and if $p>7$, then $wt(c)\geq (12q+18)/7$.
\end{theorem}

We summarize the results on the minimum weight of $C(PG(n,q))^\bot$ in the following table.

\begin{center}
\begin{tabular}{|c|c|c|}\hline
$p$ & $h$ & $d$\\
\hline
$2$ & $h$ & $2^h+2$\\
$p$ & $1$ & $2p$ \\
$7$ & $h$ & $(12q+7)/7 \leq d \leq 2q+1-(q-1)/(p-1)$\\
$p>7$ & $h$ &   $(12q+18)/7\leq d\leq2q+1-(q-1)/(p-1)$\\ \hline
\end{tabular}

\vspace{0,3 cm}

{\sl Table 1:  The minimum weight $d$ of $C(PG(n,q))^\bot$, $q=p^h$, $p$ prime, $h\geq 1$}
\end{center}
\textbf{Acknowledgement:}
The authors wish to thank the referees for their detailed reading of this article and for their valuable suggestions for improvement.

Address of the authors:\\

 Ghent University, Dept. of Pure Mathematics and Computer Algebra, Krijgslaan 281-S22, 9000 Ghent, Belgium\\
 
 \begin{tabular}{lll}
 Michel Lavrauw: & ml@cage.ugent.be & http://cage.ugent.be/$\sim$ml\\
Leo Storme: & ls@cage.ugent.be&  http://cage.ugent.be/$\sim$ls\\
Geertrui Van de Voorde: & gvdvoorde@cage.ugent.be & \\
\end{tabular}

\end{document}